\documentclass[11pt]{amsart}
\usepackage{amsmath,amssymb,amsthm,enumerate}
\usepackage[utf8]{inputenc}
\usepackage[english]{babel}
\usepackage{xy} %for diagrams
\usepackage{mathrsfs} %for script fonts
\usepackage{xspace}
\usepackage{tikz}
\usepackage{adjustbox,booktabs}
\usepackage[pagebackref=true, colorlinks]{hyperref}
\usepackage{mathtools}
\usepackage{verbatim}
\title[A new lattice invariant]{A new lattice invariant for lattices in totally disconnected locally compact groups}

\author[B. Duchesne]{Bruno Duchesne}\thanks{B.D. is supported in part by French projects ANR-14-CE25-0004 GAMME and ANR-16-CE40-0022-01 AGIRA}
\author[R. Tucker-Drob]{Robin Tucker-Drob}
\author[P. Wesolek]{Phillip Wesolek}
\date{October 2018}

\newtheorem{thm}{Theorem}[section]

\newtheorem{prop}[thm]{Proposition}
\newtheorem{lem}[thm]{Lemma}
\newtheorem{cor}[thm]{Corollary}

\theoremstyle{definition}
\newtheorem{defn}[thm]{Definition}

\newtheorem{rmk}[thm]{Remark}
\newtheorem*{claim}{Claim}
\newtheorem*{ack}{Acknowledgments}

\newcommand{\Zb}{\mathbb{Z}}
\newcommand{\Nb}{\mathbb{N}}
\newcommand{\Rb}{\mathbb{R}}
\newcommand{\Qb}{\mathbb{Q}}

\newcommand{\mc}[1]{\mathcal{#1}}

\newcommand{\tdlc}{t.d.l.c.\@\xspace}

\DeclareMathOperator{\Aut}{Aut}

\newcommand{\normal}{\trianglelefteq}

\newcommand{\con}{\mathrm{con}}
\newcommand{\nub}{\mathrm{nub}}
\newcommand{\rk}{\mathrm{rk}}
\newcommand{\FC}{\mathrm{FC}}
\newcommand{\cFC}{\ol{\mathrm{FC}}}
\newcommand{\BC}{\mathrm{BC}}

\newcommand{\RLE}{\mathrm{Rad}_{\mc{RC}}}

\newcommand{\defeq}{\coloneqq}

\newcommand{\grp}[1]{\langle #1 \rangle}
\newcommand{\ol}[1]{\overline{#1}}

\begin{document}

\begin{abstract} We introduce and explore a natural rank for totally disconnected locally compact groups called the bounded conjugacy rank. This rank is shown to be a lattice invariant for  lattices in sigma compact totally disconnected locally compact groups; that is to say, for a given sigma compact totally disconnected locally compact group, some lattice has bounded conjugacy rank $n$ if and only if every lattice has bounded conjugacy rank $n$. Several examples are then presented.
\end{abstract}
\maketitle

\section{Introduction}

A property $P$ of groups is called a \textbf{lattice invariant} if whenever $\Delta$ and $\Gamma$ are lattices in a locally compact group $G$, then $\Delta$ has $P$ if and only if $\Gamma$ has $P$. Lattice invariants are closely related to invariants of measure equivalence, where two discrete groups $\Gamma$ and $\Delta$ are called \textbf{measure equivalent} if they admit commuting actions on a non-zero sigma-finite standard measure space with each action admitting a finite measure fundamental domain; see \cite{Fu11}.  Indeed, two lattices in the same locally compact group are measure equivalent, so any invariant of measure equivalence is also a lattice invariant.

Amenability and Kazhdan's property (T) are two fundamental examples of invariants of measure equivalence and so lattice invariants. However, natural examples of lattice invariants appear to be rather rare and examples which are not additionally measure equivalence invariants are all the more rare.

In the work at hand, we discover  and explore a new lattice invariant for lattices in totally disconnected locally compact (\tdlc) groups, which we call the bounded conjugacy rank. As will become clear, bounded conjugacy rank additionally fails to be an invariant of measure equivalence. This rank appears to be a natural and robust numerical invariant which warrants further exploration.

\begin{rmk} For our main theorems, we restrict our attention to  sigma compact locally compact groups. Sigma compact locally compact groups cover most locally compact groups of interest. In particular, every compactly generated locally compact group is sigma compact.
\end{rmk}

\subsection{Statement of results}
For $G$ a topological group and $H$ a subgroup, the \textbf{BC-centralizer} of $H$ in $G$ is
\[
\BC_G(H)\defeq \{g\in G\mid g^H\text{ is relatively compact}\},
\]
where $g^H$ denotes the set of $H$ conjugates of $g$. For $G$ a \tdlc group, a closed subgroup $A\leq G$ is said to be \textbf{thin on $G$} if $A$ normalizes a compact open subgroup of $G$ and $\BC_G(A)$ is a finite index open subgroup of $G$. The {\bf thin rank} of $A$, denoted by $\rk _t(A)$, is the supremum of the ranks of free abelian subgroups of $A/\RLE(A)$, where $\RLE(A)$ is the regionally compact radical.

\begin{defn}
For $G$ a \tdlc group, the {\bf bounded conjugacy rank} of $G$, denoted by $\rk_{BC}(G)$, is the supremum of $\rk _t(A)$ as $A$ ranges over all compactly generated closed subgroups of $G$ that are thin on $G$.
\end{defn}

We show that the bounded conjugacy rank is invariant under passing to both cocompact closed subgroups and finite covolume closed subgroups, so in particular, it is a lattice invariant for \tdlc groups.

\begin{thm}[Theorem~\ref{thm:rkcovol}]\label{thm:rkcovol_intro}
Suppose that $G$ is a sigma compact \tdlc group and $H$ is a closed subgroup of $G$.
\begin{enumerate}
\item If $H$ has finite covolume in $G$, then $\rk _{BC}(H)=\rk _{BC}(G)$.
\item If $H$ is cocompact in $G$, then $\rk _{BC}(H)=\rk _{BC}(G)$.
\end{enumerate}
In particular, if $\Gamma$ and $\Delta$ are any two lattices in $G$, then
\[
\rk _{BC}(\Gamma ) = \rk _{BC}(G)=\rk _{BC}(\Delta ) .
\]
\end{thm}
Example~\ref{ex:rkcovol} shows that one cannot naively extend Theorem~\ref{thm:rkcovol_intro} to all locally compact groups.

We go on to explore the properties of discrete groups $\Gamma$, e.g.\ lattices, with non-zero bounded conjugacy rank. For a subgroup $H$ of $\Gamma$, the \textbf{FC-centralizer} of $H$ in $\Gamma$ is the subgroup $\FC _{\Gamma}(H)\defeq \{ \gamma \in \Gamma \mid \gamma ^H \text{ is finite} \}$. For finitely generated groups, the bounded conjugacy rank has a straightforward characterization related to the classical notion of the $\FC$-center of $\Gamma$ - i.e.\ the subgroup $\FC _{\Gamma}(\Gamma )$.

\begin{prop}[Proposition~\ref{prop:rk_discrete}]
Let $\Gamma$ be a discrete group.
\begin{enumerate}
\item The value  $\rk _{BC}(\Gamma )$ equals the supremum of the ranks of finitely generated free abelian subgroups $A$ of $\Gamma$ for which the index of $\FC _{\Gamma}(A)$ in $\Gamma$ is finite.

\item If $\Gamma$ is a finitely generated group, then $\rk _{BC}(\Gamma )$ is the supremum of the ranks of finitely generated free abelian subgroups $A$ of $\Gamma$ which are contained in the $\FC$-center of $\Gamma$.
\end{enumerate}
\end{prop}

The natural generalization to compactly generated \tdlc groups of part (2) of this proposition fails: Example \ref{ssec:pos_BC_rank_ex} exhibits a compactly generated \tdlc group with positive $\BC$-rank but with no non-trivial $\overline{\mathrm{FC}}$-elements.

Our work concludes with several additional examples. In particular, any discrete group with positive bounded conjugacy rank must be inner amenable (Proposition~\ref{prop:rk_discrete_inner_am}), where a group is \textbf{inner amenable} if it admits an atomless mean which is invariant under conjugation. Example~\ref{ex:lattice invariant}, due to P.E.\ Caprace, shows that inner amenability is not itself a lattice invariant, even for lattices in sigma compact \tdlc groups.

\begin{ack}
We would like to first thank Yair Hartman for his many contributions to this project. We thank Pierre-Emmanuel Caprace for explaining Example~\ref{ex:lattice invariant} to us and for allowing it to be included in the present work. We also thank Clinton Conley for discussions leading to Example \ref{ex:rkcovol}. \end{ack}

\section{Preliminaries}

We fix several notations, which will be used throughout. For $G$ a group, $H\leq G$, and $g\in G$, we write
\[
g^H\defeq \{hgh^{-1}\mid h\in H\}.
\]
If $G$ is a topological group, $\mathrm{Aut}(G)$ denotes the group of all topological group automorphisms of $G$. For $\mc{H}\subseteq \Aut(G)$ and $g\in G$, we write
\[
\mc{H}(g)\defeq \{\alpha(g)\mid \alpha\in \mc{H}\}.
\]
We use t.d.\ and l.c.\ to denote ``totally disconnected" and ``locally compact", respectively.

\subsection{Generalities on locally compact groups}\label{sec:prelim_lc}A locally compact group $G$ is \textbf{sigma compact} if it is a countable union of compact sets. It is easy to see that every compactly generated locally compact group is sigma compact, so this condition is rather mild. Every second countable locally compact group  is also sigma compact.

A closed subgroup $H$ of $G$ is said to be \textbf{cocompact} if the coset space $G/H$ is compact. We say that $H$ is of \textbf{finite covolume} if the coset space $G/H$ admits a $G$-invariant Borel probability measure, where $G$ acts on $G/H$ by left multiplication. A \textbf{lattice} of $G$ is a discrete subgroup of finite covolume.

We will require several basic facts about finite covolume subgroups and will often use these implicitly.
 \begin{thm}\label{thm:covolume}
	Suppose that $G$ is a locally compact group with $H$ and $K$ closed subgroups of $G$ such that $H\leq K \leq G$.
	\begin{enumerate}
	\item If $H$ is unimodular and cocompact in $G$, then $H$ is of finite covolume in $G$.
	\item The group $H$ is of finite covolume in $G$ if and only if $H$ is of finite covolume in $K$ and $K$ is of finite covolume in $G$.
	\end{enumerate}
\end{thm}
\begin{proof}
Claim (1) is given by \cite[Corollary B.1.8]{BeDhVa08}. Claim (2) is \cite[Lemma 1.6]{Rag72}.
\end{proof}

\begin{prop}\label{prop:KHcovol}
Suppose that $G$ is a sigma compact locally compact group with $H$ and $K$ closed subgroups such that $K\normal G$ and $G=HK$
\begin{enumerate}
	\item If $H\cap K$ is compact and $H$ is of finite covolume in $G$, then $K$ is compact.
	\item If  $H\cap K$ is compact and $H$ is cocompact in $G$, then $K$ is compact.
	\end{enumerate}
\end{prop}

\begin{proof} The map $\varphi : K/(H\cap K)\rightarrow G/H$ defined by $\varphi (x)=xH$ for $x\in K/(H\cap K)$ is continuous, bijective, and $K$-equivariant. We now argue that $\varphi$ is in fact a homeomorphism. It suffices to show the composition $\varphi ^{-1}\circ \pi _{G/H}$ is continuous, where $\pi _{G/H}:G\rightarrow G/H$ is the quotient map.

Since $K$ is normal in $G=HK$, the map $\psi : H/(H\cap K) \rightarrow G/K$, defined by $\psi (y)= yK$ for $y\in H/(H\cap  K)$, is a continuous isomorphism of sigma compact locally compact groups, hence it is a homeomorphism; see \cite[Theorem 5.33]{HeKe79}. The composition $\psi ^{-1} \circ \pi _{G/K} : G\rightarrow H/(H\cap K)$ is a continuous homomorphism to the group $H/(H\cap K)$. Since $G$ acts on $G/(H\cap K)$ continuously by left multiplication, the map $G\to K/(H\cap K)$ by
\[
g\mapsto g^{-1}\psi ^{-1}\circ\pi _{G/K}(g)
\]
is continuous. Composing with $\varphi$ yields
\[
\varphi (g^{-1}\psi ^{-1}\circ \pi _{G/K}(g))=g^{-1}\psi ^{-1}\circ\pi _{G/K}(g)H = g^{-1}H = \pi _{G/H}(g^{-1}) .
\]
We deduce that $\varphi ^{-1}\circ \pi _{G/H} (g^{-1}) = g^{-1}\psi ^{-1}\circ \pi _{G/K}(g)$, which shows that $\varphi ^{-1}\circ \pi _{G/H}$ is continuous. Hence, $\varphi$ is a homeomorphism.

For (1), suppose $H$ is of finite covolume in $G$ and let $\mu$ be a $G$-invariant Borel probability measure on $G/H$. Since $\varphi$ is a homeomorphism,  the push forward $(\varphi ^{-1})_*\mu$ is then a $K$-invariant Borel probability measure on $K/(H\cap K)$, so $H\cap K$ is a compact subgroup of $K$ of finite covolume in $K$. Applying Theorem~\ref{thm:covolume}, the trivial subgroup is of finite covolume in $K$, and this implies that the Haar measure on $K$ is finite. Hence, $K$ is compact, establishing (1).

Claim (2) is similar.
\end{proof}

A locally compact group $G$ is called \textbf{regionally compact} if for every finite set $F\subseteq G$, the subgroup $\grp{F}$ is relatively compact. This property is sometimes called ``locally elliptic" or ``topologically locally finite"; see \cite[Remark 1.0.1]{CRW17} for an explaination of this terminology change. By classical work of V.P. Platonov \cite{Plat65}, every locally compact group $G$ admits a unique largest regionally compact normal subgroup called the \textbf{regionally compact radical} and denoted by $\RLE(G)$. Furthermore, $\RLE(G/\RLE(G))=\{1\}$. Proofs of these facts can be found in \cite[Section 2]{Cap09}

A locally compact group $G$ is said to be $\cFC$ if $g^G$ is relatively compact for every $g\in G$. An abstract group is called $\FC$ is every element has a finite conjugacy class. The following result is also classical, but we give a proof for completeness.

\begin{thm}[Usakov, \cite{Us63}]\label{thm:usakov} Suppose that $G$ is a compactly generated \tdlc group. If $G$ is an $\cFC$ group, then $\RLE(G)$ is compact and $G/\RLE(G)$ is free abelian.
\end{thm}
\begin{proof}
Let $P(G)$ be the collection of elements $g\in G$ such that $\ol{\grp{g}}$ is compact. By \cite[Corollary 1.5]{Wan71}, any compact normal set $K\subseteq P(G)$ is such that $\grp{K}$ is relatively compact. Each $g\in P(G)$ is then an element of a compact normal subgroup of $G$, since $G$ is $\cFC$. We deduce that the regionally compact radical contains $P(G)$ and is open.

The group $\tilde{G}\defeq G/\RLE(G)$ is a discrete group and every element has a compact conjugacy class. Every element of $\tilde{G}$ indeed has a finite conjugacy class, so $\tilde{G}$ is a finitely generated $\FC$ group. The group $\tilde{G}$ is also finitely generated, so the center of $\tilde{G}$ is of finite index. Therefore, $[\tilde{G},\tilde{G}]$ is finite by the Schur--Baer Theorem. On the other hand, $\tilde{G}$ must have a trivial regionally compact radical, so $[\tilde{G},\tilde{G}]$ is trivial. We conclude that $\tilde{G}$ is free abelian.

It remains to show that $\RLE(G)$ is compact. Let $P$ be the collection of finite subsets of $\RLE(G)$. The collection $I\defeq \{\ol{\grp{F^G}}\mid F\in P\}$ is a directed system of compact normal subgroups of $G$, and $\ol{\bigcup I}=\RLE(G)$. Applying \cite[Theorem 1.9]{ReWe17}, there is  $N\in I$ and $K\normal G$ closed such that $K\leq \RLE(G)$, $K/N$ is compact, and $\RLE(G)/K$ is discrete. Since $N$ is compact, the subgroup $K$ is outright compact. The quotient $G/K$ is then a discrete $\cFC$ group, so it is an $\FC$ group. Arguing as in the second paragraph yields a closed normal subgroup $K'\normal G$ such that $K'/K$ is finite and $G/K'$ is torsion free abelian. We infer that $K'=\RLE(G)$, and thus $\RLE(G)$ is compact.
\end{proof}

\subsection{Willis theory}
Let $G$ be a \tdlc group. The \textbf{scale function} $s:\mathrm{Aut}(G)\rightarrow \Zb$ is defined by
\[
s(\alpha )\defeq \min\{|U:U\cap \alpha ^{-1}(U)|:U\text{ is compact and open}\}
\]
A compact open subgroup that witnesses the value $s(\alpha )$ is called a \textbf{tidy} subgroup for $\alpha$. A compact open subgroup $U$ is tidy for $\alpha$ if and only if it is tidy for $\alpha ^{-1}$; consider \cite[Corollary to Lemma 3]{Wi94}. If $s(\alpha )=1$ and $s(\alpha ^{-1})=1$, it then follows that there is a compact open subgroup of $G$ which is $\alpha$-invariant. In this case, the automorphism $\alpha$ is called \textbf{uniscalar}.

There are three important subgroups associated to $\alpha \in \Aut(G)$ and related to the scale.  The first is the \textbf{contraction group} of $\alpha$ in $G$:
\[
\con(\alpha )\defeq \{ x\in G\mid \alpha ^n (x) \rightarrow 1\text{ as }n\rightarrow +\infty\}.
\]
For $g\in G$, we write $\con(g)$ for the contraction group given by regarding $g$ as an inner automorphism.

The contraction group gives a way to isolate uniscalar automorphisms.
\begin{prop}[Baumgartner--Willis, {\cite[Proposition 3.24]{BW04}}]\label{prop:BW_uniscalar}
	For $G$ a \tdlc group and $\alpha \in \mathrm{Aut}(G)$, the following are equivalent:
\begin{enumerate}[(1)]
	\item $s(\alpha )=1$.
	\item $\con(\alpha ^{-1})$ is relatively compact.
\end{enumerate}
\end{prop}

Contraction groups are not closed in general; for instance, take the contraction group for the unit shift automorphism of $F^{\Zb}$ where $F$ is some non-trivial finite group. We use the notation $\ol{\con}(\alpha )$ for the closure of the contraction group.

The \textbf{nub} of an automorphism $\alpha$ of $G$ is
\[
\nub(\alpha )\defeq \bigcap\{U\mid U\text{ tidy for }\alpha \} .
\]
Lastly,
\[
\mathrm{P}_G(\alpha)\defeq  \{ x\in G \mid \{ \alpha ^n (x) \} _{n\geq 0} \text{ is relatively compact}\}.
\]
For $g\in G$, we write $\nub(g)$ and $\mathrm{P}_G(g)$, where $g$ is considered as an inner automorphism.

Let us note several properties of these subgroups.

\begin{prop}[Baumgartner--Willis, Willis]\label{prop:ccontraction-1}
	For $G$ a \tdlc group and $\alpha\in \mathrm{Aut}(G)$,
\begin{enumerate}
\item $\mathrm{P}_G(\alpha )$ is a closed subgroup of $G$,
\item $\ol{\con}(\alpha )=\con(\alpha )\nub(\alpha )$,
\item $\mathrm{con}(\alpha )$ is a normal subgroup of $\mathrm{P}_G(\alpha )$,
\end{enumerate}
\end{prop}

\begin{proof}
(1) is \cite[Proposition 3]{Wi94}; (2) is \cite[Corollary 3.30]{BW04}; and (3) is \cite[Proposition 3.4]{BW04}.
\end{proof}

The theory of the scale function admits an extension to subgroups of automorphisms, instead of single automorphisms.
\begin{defn}
A subgroup $\mathcal{H}\leq \mathrm{Aut}(G)$ is said to be \textbf{flat} on $G$ if there exists a compact open subgroup $U$ of $G$ which is simultaneously tidy for every $\alpha \in \mathcal{H}$. In this case, $U$ is said to be \textbf{tidy for $\mathcal{H}$}. A subgroup $H$ of $G$ is \textbf{flat} on $G$ if it is flat as a collection of inner automorphims. Flat groups of automorphisms first appear in \cite{Wi04}; the terminology used here is introduced in \cite{SW13}.
\end{defn}

Flat groups of automorphisms admit a canonical normal subgroup.
\begin{defn}
Let $\mathcal{H}\leq \mathrm{Aut}(G)$ be a flat group of automorphisms of a \tdlc group $G$. The \textbf{uniscalar subgroup} of $\mathcal{H}$ is the group $\mathcal{H}_u \defeq  \{ \alpha \in \mathcal{H} \mid s(\alpha ) =1 =s(\alpha ^{-1}) \}$. We say that $\mathcal{H}$ is \textbf{uniscalar} if $\mathcal{H}=\mathcal{H}_u$.
\end{defn}

The subgroup $\mathcal{H}_u$ is the collection of elements of $\mathcal{H}$ which normalize some compact open subgroup of $G$. The group $\mathcal{H}_u$ is a normal subgroup of $\mathcal{H}$. By \cite[Corollary 6.15]{Wi04}, $\mathcal{H}/\mathcal{H}_u$ is a free abelian group, and furthermore a strong converse holds.

\begin{thm}[{Shalom--Willis, \cite[Theorems 4.9 and 4.13]{SW13}}]\label{thm:common_tidy_sgrp}
Let $G$ be a \tdlc group and $\mathcal{N}\triangleleft \mathcal{H}\leq \mathrm{Aut}(G)$ with $\mathcal{N}$ normal in $\mathcal{H}$. Assume that there is a compact open subgroup $V$ of $G$ which is $\mathcal{N}$-invariant.
\begin{enumerate}
\item If $\mathcal{H}/\mathcal{N}$ is a finitely generated nilpotent group, then $\mathcal{H}$ is flat.
\item If $\mathcal{H}/\mathcal{N}$ is a finitely generated polycyclic group, then $\mathcal{H}$ has a finite index subgroup which is flat.
\end{enumerate}
\end{thm}

\subsection{BC-centralizers}
For $G$ a \tdlc group, the \textbf{BC-centralizer} of $\mathcal{H}\leq \mathrm{Aut}(G)$ in $G$ is
\[
\BC_G(\mathcal{H})\defeq \{g\in G\mid \mathcal{H}(g)\text{ is relatively compact}\}.
\]
For a subgroup $H$ of $G$, we may define
\[
\BC_G(H)\defeq \{g\in G\mid g^H\text{ is relatively compact}\}.
\]
The definition of the BC-centralizer for subgroups is the obvious restatement obtained by regarding a subgroup as a group of inner automorphisms. The set $\BC_G(\mathcal{H})$ is a subgroup of $G$, but it need not be closed in general.  For $\alpha \in \mathrm{Aut}(G)$, we write  $\BC _G(\alpha )$ for $\BC _G(\langle \alpha \rangle )$, and for $g\in G$, we write $\BC _G(g)$ for $\BC _G(\langle g\rangle )$.

The following proposition, established in \cite{Wi94} and \cite{BW04}, gives the relevant connections between the BC-centralizer of a cyclic group and the previously mentioned subgroups associated to an automorphism.

\begin{prop}[Baumgartner--Willis, Willis]\label{prop:ccontraction}
	For $G$ a \tdlc group and $\alpha\in \mathrm{Aut}(G)$,
\begin{enumerate}
\item  $\BC _G(\alpha )$ is a closed subgroup of $G$,
\item $\ol{\con}(\alpha )\cap\BC _G(\alpha ) =\nub (\alpha )$,
\item $\mathrm{con}(\alpha )\BC _G(\alpha ) = \mathrm{P}_G(\alpha )$.
\end{enumerate}
\end{prop}

\begin{proof}
(1) is \cite[Proposition 3]{Wi94}; (4) is \cite[Lemma 3.29]{BW04}; and (5) is \cite[Corollary 3.17]{BW04}.
\end{proof}

Proposition~\ref{prop:ccontraction} ensures $\BC _G(\alpha )$ is closed for each $\alpha \in \mathrm{Aut}(G)$. Willis' proof of the relevant claim of Proposition~\ref{prop:ccontraction} in fact shows that $\BC _G(\mathcal{H})$ is closed whenever $\mathcal{H}$ is flat. We give a proof for completeness.

\begin{prop}[Willis]\label{prop:flatclosed}
Let $G$ be a \tdlc group and $\mathcal{H}\leq \mathrm{Aut}(G)$ be a flat group of automorphisms of $G$.
\begin{enumerate}
\item $\BC _G(\mathcal{H})$ is a closed subgroup of $G$.
\item If $U$ is a compact open subgroup of $G$ which is tidy for $\mathcal{H}$, then $U\cap \BC _G(\mathcal{H})= \bigcap _{\alpha \in \mathcal{H}}\alpha (U)$.
\end{enumerate}
\end{prop}

\begin{proof}
Let $U$ be a compact open subgroup of $G$ which is tidy for $\mathcal{H}$. That $\BC _G(\mathcal{H})$ is closed follows from $U\cap \BC _G(\mathcal{H})$ closed. It therefore suffices to prove (2). Since $U$ is tidy for $\mathcal{H}$, \cite[Lemma 9]{Wi94} ensures that $U\cap \BC _G(\mathcal{H})\leq \bigcap _{\alpha \in \mathcal{H}}\alpha (U)$. The reverse containment is clear.
\end{proof}

We can characterize uniscalar flat groups via the BC-centralizer.

\begin{prop}\label{prop:uniscalar}
For $\mathcal{H}$ a flat group of automorphisms of a \tdlc group $G$,  $\mathcal{H}$ is uniscalar if and only if $\BC _G(\mathcal{H})$ is open in $G$.
\end{prop}

\begin{proof}
Fix a tidy subgroup $U$ for $\mathcal{H}$. If $\mathcal{H}$ is uniscalar, then $\mathcal{H}$ normalizes $U$, $U\leq \BC _G(\mathcal{H})$, and $\BC _G(\mathcal{H})$ is open. Conversely, if $\BC _G(\mathcal{H})$ is open, then $\BC _G(\alpha )$ is open for every $\alpha \in \mathcal{H}$. From Proposition~\ref{prop:flatclosed}, it now follows that $s(\alpha ) = 1$ for every $\alpha \in \mathcal{H}$. That is to say, $\mathcal{H}$ is uniscalar.
\end{proof}

\section{Thin subgroups}
We begin by defining a special family of flat subgroups via which we will isolate the bounded conjugacy rank. These definitions can be made for groups of automorphisms, but we restrict to subgroups of the ambient group for clarity and brevity.

\begin{defn}
For $G$ a \tdlc group, a subgroup $A\leq G$ is called \textbf{thin} on $G$ if it is flat on $G$ and $\BC _G(A)$ is a finite index open subgroup of $G$. That this definition agrees with the definition given in the introduction follows from the next corollary.
\end{defn}

As an immediate consequence of Proposition \ref{prop:uniscalar} and the definition, thin groups of automorphisms are uniscalar.
\begin{cor}\label{cor:thin}
For $G$ a \tdlc group, if $A$ is a thin subgroup of $G$, then $A$ is uniscalar.
\end{cor}

Thin subgroups form a robust subfamily of the flat subgroups.

\begin{lem}\label{lem:flatcovol}
For $A$ a flat subgroup of sigma compact \tdlc group $G$, the following are equivalent:
\begin{enumerate}
\item $\BC _G (A)$ is of finite covolume in $G$,
\item $\BC _G (A)$ is cocompact in $G$,
\item $\BC _G (A)$ is finite index in $G$,
\item $A$ is thin on $G$.
\end{enumerate}
\end{lem}

\begin{proof}
By Proposition \ref{prop:flatclosed}, $\BC _G(A)$ is closed, so the equivalence of (3) and (4) is immediate. The implications (4)$\Rightarrow$(1) and (4)$\Rightarrow$(2) are also clear.

For (1)$\Rightarrow$(4), suppose that $\BC _G (A)$ is of finite covolume in $G$. For any $a\in A$, Proposition \ref{prop:ccontraction} ensures that $\ol{\con}(a ) \BC _G(a ) = \mathrm{P}_G(a)$ and that $\ol{\con}(a ) \cap \BC _G (a )=\nub (a)$. The subgroup $\nub(a)$ is compact, and since $\BC _G(A)\leq \BC _G(a )$, the group $\BC _G(a )$ is also of finite covolume in $\mathrm{P}_G(a )$. Proposition \ref{prop:KHcovol} thus implies that $\ol{\con}(a )$ is compact. Similarly, $\ol{\con}(a^{-1})$ is compact. Appealing to Proposition \ref{prop:BW_uniscalar}, we deduce that $s(a ) = s(a ^{-1})=1$. The group $A$ is thus flat and uniscalar, hence $A$ normalizes some compact open subgroup $U$ of $G$. The subgroup $U$ is contained in $\BC _G(A)$, so $\BC _G (A)$ is open and of finite covolume in $G$. We infer that $\BC_G(A)$ is of finite index in $G$.

The same argument gives the implication (2)$\Rightarrow$(4).
\end{proof}

The property of being a thin subgroup also enjoys a hereditary property. The proof is immediate from the definitions.
\begin{lem}\label{lem:hereditary}
Let $G$ be a \tdlc group with $H$ and $A$ closed subgroups of $G$. If  $A$ is thin on $G$ and $A\leq H$, then $A$ is thin on $H$. In particular, if $A\leq G$ is a closed subgroup that is thin on $G$, then $A$ is thin on itself.
\end{lem}

Thin subgroups finally have a well-understood internal structure.

\begin{lem}\label{lem:KL}
Let $A$ be a compactly generated closed subgroup of a \tdlc group $G$. If $A$ is thin on $G$, then $\RLE(A)$ is compact and relatively open in $A$, and $A/\RLE(A)$ has a finite index free abelian subgroup.
\end{lem}

\begin{proof}
 Consider $B\defeq \BC_A(A)=\BC_G(A)\cap A$. The group $B$, being a finite index relatively open subgroup of $A$, is compactly generated, and $B$ is an ${\ol{\mathrm{FC}}}$ group. Applying Theorem~\ref{thm:usakov}, $\RLE(B)$ is a compact open normal subgroup of $B$, and $B/\RLE(B)$ is isomorphic to $\Zb ^n$ for some $0\leq n<\infty$.  Since $\RLE(B)$ is topologically characteristic in $B$, we deduce that $A/\RLE(B)$ is virtually free abelian.

It remains to show that $\RLE(B)=\RLE(A)$.  Clearly, $\RLE(B)\leq \RLE(A)$. On other hand, $\RLE(A)\cap B=\RLE(B)$, so $|\RLE(A):\RLE(B)|<\infty$. The subgroup $\RLE(A)$ is then compact, and therefore, $\RLE(A)\leq B$. Hence, $\RLE(A)=\RLE(B)$.
\end{proof}

We conclude by isolating the thin subgroups among the compactly generated closed subgroups.

\begin{prop}\label{prop:thin_char}
Suppose that $A$ is a compactly generated closed subgroup of a \tdlc group $G$. Then $A$ is thin if and only if $\BC _G(A)$ is of finite index and open in $G$.
\end{prop}
\begin{proof}
The non-trivial implication is the reverse, so let us assume that $\BC _G(A)$ is of finite index and open in $G$. We argue that $A$ is flat.

As in the proof of Lemma~\ref{lem:KL}, $A/\RLE(A)$ is virtually free abelian, and $\RLE(A)$ is compact. We may thus find $B\normal A$ of finite index such that $\RLE(A)\leq B$ and $B/\RLE(A)$ is free abelian. Applying Theorem \ref{thm:common_tidy_sgrp}, we conclude that $B$ is flat, and it follows from Proposition \ref{prop:uniscalar} that $B$ is uniscalar. Let $U$ be a compact open subgroup normalized by $B$. Since $B$ is of finite index in $A$, the group $U$ has only finitely many conjugates under $A$, and the intersection of all of these conjugates is a compact open subgroup of $G$ which is normalized by $A$. The subgroup $A$ is thus uniscalar, so in particular, it is flat.
\end{proof}

\section{Bounded conjugacy rank}\label{sec:BCR}

\subsection{Lattice invariant}
We are now prepared to prove Theorem~\ref{thm:rkcovol_intro}. In view of Lemma~\ref{lem:KL}, the next definition is sensible. Recall that the rank of a free abelian group is the minimum number of generators.
\begin{defn}
For $A$ a compactly generated \tdlc group that is thin on itself, the {\bf thin rank} of $A$, denoted by $\rk _t(A)$, is the maximum of the ranks of free abelian subgroups of $A/\RLE(A)$.
\end{defn}
Via the thin rank, we define the desired rank.
\begin{defn}
For $G$ a \tdlc group, the {\bf bounded conjugacy rank} of $G$, denoted by $\rk_{BC}(G)$, is the supremum of $\rk _t(A)$ as $A$ ranges over all compactly generated closed subgroups that are thin on $G$.
\end{defn}
Lemma~\ref{lem:hereditary} ensures that each $A\leq G$ that is thin on $G$ is also thin on itself, so the bounded conjugacy rank is well-defined. Note further that the bounded conjugacy rank can be infinite; see Example~\ref{ex:M-groups}.

The proof of the desired theorem requires a technical lemma.

\begin{lem}\label{lem:sgrp_thin_grp}
For $A$ a compactly generated \tdlc group that is thin on itself and $B$ a closed subgroup of $A$, the following are equivalent:
\begin{enumerate}
\item $B$ is cocompact in $A$.
\item $B$ is of finite covolume in $A$.
\item $B\RLE(A)$ is of finite index and open in $A$.
\item $\rk _t (B) = \rk _t (A)$.
\end{enumerate}
\end{lem}

\begin{proof}
The equivalence of (1) and (3) and (2)$\Rightarrow$(3) are immediate. For (3)$\Rightarrow $(2), it is immediate that $B$ is cocompact. Lemma~\ref{lem:KL} ensures that $A$ and $B$ have a compact open normal subgroup, and thus, they are both unimodular. Lemma~\ref{thm:covolume} then implies that $B$ is of finite covolume.

(3)$\Rightarrow$(4). Suppose that $C\defeq B\RLE(A)$ is of finite index and open in $A$. One verifies that $\RLE(C)=\RLE(A)$, so $C/\RLE(C)$ is isomorphic to a finite index subgroup of $A/\RLE(A)$. We conclude that $\rk _t (C) = \rk _t (A)$. On the other hand, the projection $\pi:B\rightarrow C/\RLE(C)$ is onto with kernel $\RLE(C)\cap B$, and $\pi(\RLE(B))$ lies in the regionally compact radical of $C$. We deduce that $\RLE(C)\cap B=\RLE(B)$. The quotient $B/\RLE(B)$ is thus isomorphic to $C/\RLE(C)$, so $\rk_t(B)=\rk_t(C)$. Claim (4) is now demonstrated.

(4)$\Rightarrow$(3). The group $B/\RLE(A)\cap B$ is isomorphic to the subgroup $B\RLE(A)/\RLE(A)$ of $A/\RLE(A)$. The group $B/\RLE(A)\cap B$ is thus virtually free abelian, and the maximum rank of free abelian subgroups is less than or equal to $\rk_t(A)$. On the other hand, $B/\RLE(A)\cap B$  maps onto $B/\RLE(B)$, so the maximum rank of free abelian subgroups of $B/\RLE(A)\cap B$ is greater than or equal to the maximum rank of free abelian subgroups of $B/\RLE(B)$. Since $\rk_t(B)=\rk_t(A)$, it now follows that the maximum rank of free abelian subgroups of $B/\RLE(A)\cap B$ equals $\rk_t(A)$. We conclude that $B\RLE(A)/\RLE(A)$ has the same maximum rank of free abelian subgroups as $A/\RLE(A)$, so $B\RLE(A)/\RLE(A)$ has finite index in $A$, verifying (3).
\end{proof}

\begin{thm}\label{thm:rkcovol}
Suppose that $G$ is a sigma compact \tdlc group and $H$ is a closed subgroup of $G$.
\begin{enumerate}
\item If $H$ has finite covolume in $G$, then $\rk _{BC}(H)=\rk _{BC}(G)$.
\item If $H$ is cocompact in $G$, then $\rk _{BC}(H)=\rk _{BC}(G)$.
\end{enumerate}
In particular, if $\Gamma$ and $\Delta$ are any two lattices in $G$, then
\[
\rk _{BC}(\Gamma ) = \rk _{BC}(G)=\rk _{BC}(\Delta ) .
\]
\end{thm}

\begin{proof}
Suppose that $H$ is cocompact (of finite covolume) in $G$. Let $A$ be any compactly generated closed subgroup of $G$ which is thin on $G$. Applying Corollary~\ref{cor:thin}, we obtain $U$ a compact open  subgroup of $G$ which is normalized by $A$. The subgroup $AU$ is thin on $G$ and open. The intersection $B\defeq AU\cap H$ is then cocompact (of finite covolume) in $AU$, so $\rk _t (B)=\rk _t(A)$ by Lemma~\ref{lem:sgrp_thin_grp}. In addition, since $B$ is thin on $G$, it is thin on $H$. We conclude that $\rk _{BC}(H)\geq \rk _t(B)=\rk _t(A)$. Taking the supremum of the $\rk_t(A)$ as $A$ varies over compactly generated subgroups that are thin on $G$, we obtain $\rk _{BC}(H)\geq \rk _{BC}(G)$.

Conversely, let $B$ be any closed compactly generated subgroup of $H$ which is thin on $H$. Applying Lemma~\ref{lem:KL}, there is a finite index open $B_0\normal B$ such that $B_0/\RLE(B)$ is a finitely generated free abelian group. Theorem \ref{thm:common_tidy_sgrp} implies that $B_0$ is flat on $G$. Moreover, $\BC _G(B_0)$ contains $\BC _H(B_0)=\BC _H(B)$, so $\BC _G(B_0)$ is cocompact (of finite covolume) in $G$. We conclude that $B_0$ is thin on $G$ by Lemma \ref{lem:flatcovol}, so $\rk  _{BC}(G)\geq \rk _t(B_0)=\rk _t(B)$. Taking the supremum over all compactly generated closed $B\leq H$ thin on $H$ yields that $\rk _{BC}(G)\geq \rk _{BC}(H)$.
\end{proof}

Theorem~\ref{thm:rkcovol} gives a new lattice invariant for lattices in sigma compact \tdlc groups. Surprisingly, Example~\ref{ex:rkcovol} shows that one cannot naively extend Theorem~\ref{thm:rkcovol} to all sigma compact locally compact groups.

\subsection{Bounded conjugacy rank in discrete groups}
The $BC$-rank for discrete groups, so for lattices in particular, has a simpler characterization than for general \tdlc groups. The characterization becomes even more straightforward for finitely generated discrete groups.

For a discrete group $\Gamma$ and a subgroup $H$ of $\Gamma$, we define $\FC _{\Gamma}(H)$ to be the set of all elements $\gamma \in \Gamma$ for which $\gamma ^H$ is finite; this is easily seen to be a subgroup of $\Gamma$. The \textbf{FC-center} of $\Gamma$ is the subgroup $\FC _{\Gamma}(\Gamma )$ - i.e.\ the collection of all group elements with a finite conjugacy class.

\begin{prop}\label{prop:rk_discrete}
Let $\Gamma$ be a discrete group.
\begin{enumerate}
\item The value $\rk _{BC}(\Gamma )$ equals the supremum of the ranks of finitely generated free abelian subgroups $A$ of $\Gamma$ for which the index of $\FC _{\Gamma}(A)$ in $\Gamma$ is finite.
\item If $\Gamma$ is a finitely generated group, then $\rk _{BC}(\Gamma )$ is the supremum of the ranks of finitely generated free abelian subgroups $A$ of $\Gamma$ which are contained in the $\FC$-center of $\Gamma$.
\end{enumerate}
\end{prop}
\begin{proof}
We note that since $\Gamma$ is discrete, BC-centralizers and FC-centralizers are the same.

For (1), suppose $A\leq \Gamma$ is finitely generated and thin on $\Gamma$. The subgroup $A\cap \FC_{\Gamma}(A)$ is of finite index in $A$, so it is finitely generated. Fix a finite generating set $F$. For each $a\in F$, the set $a^A$ is finite, so each $a\in F$ commutes with a finite index subgroup of $A$. There is thus a finite index subgroup $B$ of $A\cap \FC _{\Gamma}(A)$ which commutes with $F$ and hence with all of $A\cap \FC_{\Gamma}(A)$. The subgroup $B$ is thus an abelian finite index subgroup of $A$, and in particular, it is finitely generated.

We may find a free abelian finite index subgroup $B_0$ of $B$. The subgroup $B_0$, being a subgroup of $A$, is thin on $\Gamma$, and since $B_0$ has finite index in $A$, Lemma~\ref{lem:sgrp_thin_grp} ensures that the rank of $B_0$ coincides with $\rk_t(A)$. To compute $\rk_{BC}(\Gamma)$, it thus suffices to consider finitely generated free abelian thin subgroups of $\Gamma$.

For (2), it suffices to argue that every finitely generated free abelian thin subgroup $A$ of $\Gamma$ has a finite index subgroup contained in the $\FC$-center of $\Gamma$, in view of part (1). Given such a subgroup $A$, the subgroup $\FC_{\Gamma}(A)$ is of finite index in $\Gamma$, so it is finitely generated. Arguing as in part (1), we may find a finite index subgroup $B$ of $A$ which commutes with  $\FC _{\Gamma}(A)$. Since $B$ commutes with a finite index subgroup of $\Gamma$, it is contained in the $\FC$-center of $\Gamma$. The proposition now follows.
\end{proof}

Proposition~\ref{prop:rk_discrete} shows that every finitely generated group with positive BC-rank admits a non-locally finite FC-center. On the other hand, this does not generalize to the case of compactly generated \tdlc groups; see Example~\ref{ssec:pos_BC_rank_ex}. That is, for compactly generated \tdlc groups, positive BC-rank does not imply non-regionally compact $\cFC$-center.

We conclude this section by noting that non-zero $\BC$-rank gives a sufficient, but not necessary, condition to ensure inner amenability. One can thus think of positive $\BC$-rank as a strong algebraic form of inner amenability. In contrast to our results for $\BC$-rank, Example~\ref{ex:lattice invariant} shows that inner amenability fails to be a lattice invariant for lattices in sigma compact \tdlc groups.

\begin{prop}\label{prop:rk_discrete_inner_am}
Let $\Gamma$ be a discrete group with $\rk_{BC}(\Gamma)>0$. Then $\Gamma$ is inner amenable.
\end{prop}
\begin{proof}
Fix $A\leq \Gamma$ a finitely generated free abelian subgroup of positive rank with $\mathrm{FC}_\Gamma (A)$ of finite index in $\Gamma$; such a subgroup exist by (1) of Proposition~\ref{prop:rk_discrete} since $\rk_{BC}(\Gamma)>0$. After passing to a subgroup of $A$, we may assume that $A$ is infinite cyclic with cyclic generator $a$.

Setting $B\defeq\FC_{\Gamma}(A)$, we have that $A\leq B$, and each $b\in B$ commutes with some positive power of $a$. Each $b\in B$ thus commutes with all but finitely many terms of the sequence $(a^{n!})_n$. Letting $\delta _n$ denote the point mass at $a^{n!}$, any weak${}^*$ cluster point of the sequence $(\delta _{n})_n$ in $\ell ^{\infty}(B)^*$ is a conjugation-invariant atomless mean on $B$. Hence, $B$ is inner amenable, and since $B$ has finite index in $\Gamma$, the group $\Gamma$ is inner amenable as well by \cite[Th\'{e}or\`{e}me 1]{GdlH91}.
\end{proof}

\begin{rmk}
Theorem~\ref{thm:rkcovol} along with Proposition~\ref{prop:rk_discrete} give insight into an interesting example due to Y. Cornulier. In \cite[Section 2.B]{Cor15}, Cornulier builds an example of a compactly generated \tdlc group with cocompact lattices $\Gamma$ and $\Delta$ such that $\Gamma$ has an infinite torsion FC-center, but $\Delta$ has a trivial FC-center.  Theorem~\ref{thm:rkcovol} shows the requirement that the FC-center of $\Gamma$ is torsion is necessary.
\end{rmk}

\section{Examples}

\subsection{Merzlyakov's groups}\label{ex:M-groups}

This example exhibits a group $\Gamma$ such that $\rk_{BC}(\Gamma )=\infty$, but every abelian subgroup of $\Gamma$ is finitely generated. This shows that taking the supremum in the definition of $\rk_{BC}$ is necessary.

 We consider groups which are inductive limits with the following form. Take $\Gamma _0$ to be the trivial group. Having defined $\Gamma _n$, let $\Gamma _{n+1}$ be the semidirect product $\Gamma _{n+1}\defeq  B_n\rtimes _{\alpha _n}\Gamma _n$, where $\alpha _n : \Gamma _n\curvearrowright  B_{n}$ is an action of $\Gamma _n$ by automorphisms on a nontrivial finite rank free abelian group $B_{n}$. Let $i_n : \Gamma _n \hookrightarrow \{ 1 \} \rtimes \Gamma _n \leq \Gamma _{n+1}$ be the natural inclusion map. Define $\Gamma = \Gamma _{(\alpha _n , B_n)_{n\in \Nb}}$ to be the inductive limit of this sequence. Thus,
\begin{align*}
\Gamma _{n} &= B_{n-1}\rtimes _{\alpha _{n-1}}(B_{n-2}\rtimes _{\alpha _{n-2}}(\cdots \rtimes _{\alpha _3}(B_2 \rtimes _{\alpha _2}(B_1 \rtimes _{\alpha _1}B_0))\cdots )), \\
\Gamma &= \{ b \in \prod _{i\in \Nb} B_i \mid b_i =1 \text{ for all but finitely many }i \} ,
\end{align*}
and the product $b\circledast c$ of the group elements $b,c \in \Gamma$ is determined by the rule
\[
(b\circledast c )_n = b_n \alpha _n (p_n (b))(c_n),
\]
where $p_n : \Gamma \rightarrow \Gamma _n$ denotes the projection map with $p_n (b)_i = b_i$ for $0\leq i<n$. The injective homomorphism $\hat{i}_n :\Gamma _n \rightarrow \Gamma$ given by
\[
\hat{i}_n (b)_i =
\begin{cases}
b_i &\text{if }0\leq i <n \\
1 &\text{if }i\geq n
\end{cases}
\]
is a section for $p_n$, so $\Gamma$ is the (internal) semidirect product $\Gamma  = \ker (p_n)\rtimes \hat{i}_n(\Gamma  _n )$.

In \cite{Mer69}, Merzlyakov constructs a sequence $(B_n ,\alpha _n )_{n\in \Nb}$ as above, for which the resulting inductive limit $\Gamma \defeq  \Gamma _{(B_n,\alpha _n )_{n\in \Nb}}$ has the following properties:
\begin{enumerate}
\item Each of the groups $\Gamma_n$ is torsion free, polycyclic, and has a finite index free abelian subgroup $A_n$, with $\mathrm{rk}(A_n)\rightarrow \infty$. Hence, $\Gamma$ has free abelian subgroups of arbitrarily large rank.
\item For each $n$, the image $P_{n} \defeq  \alpha _n (\Gamma _n)$ of $\Gamma _n$ in $\mathrm{Aut}(B_{n})$ is finite.
\item Every abelian subgroup of $\Gamma$ is finitely generated. In particular, $\Gamma$ has no free abelian subgroups of infinite rank.
\end{enumerate}
(In addition, Merzlyakov shows in the follow up paper \cite{Mer84} that $\Gamma$ has finite abelian section rank.) It follows from (1) and (2) that every orbit of the conjugation action of the abelian group $\hat{i}_n(A_n)$ on the finite index subgroup $\ker (p_n )\rtimes \hat{i}_n(A_n)$ of $\Gamma$ is finite. Thus, $\mathrm{BC}_{\Gamma }(\hat{i}_n(A_n))$ has finite index in $\Gamma$, and hence $\mathrm{rk}_{BC}(\Gamma ) \geq \mathrm{rk}(A_n) \rightarrow \infty$.

\subsection{Merzlyakov's groups and lattices}\label{ex:MerzLattice} We now show that Merzlyakov's group $\Gamma$ from Example 1 and the countably infinite rank free abelian group $B\defeq  \bigoplus _{n\in \Nb} B_n$ are cocompact lattices in the same \tdlc group $G$. Thus, while the existence of a rank $n$ free abelian thin subgroup is a lattice invariant for $n<\infty$, it is not a lattice invariant for $n=\infty$.

By (2) above, $P_n=\alpha _n(\Gamma _n)$ is a finite subgroup of $\mathrm{Aut}(B_n)$ for each $n\in \Nb$. The direct product $P\defeq \prod _{n\geq 0} P_n$ is therefore a profinite group, and $P$ acts continuously on the discrete group $B$ by automorphisms, via the coordinate-wise action $(p\cdot b)_n \defeq  p_n(b_n)$ for $p\in P$, $b\in B$, and $n\in \Nb$. Consider the associated semidirect product $G \defeq  B\rtimes P$, which is a \tdlc group having $B\rtimes \{ 1_P \}$ as a discrete cocompact normal subgroup. In particular, $B$ is isomorphic to a lattice in $G$.

It remains to show that $\Gamma$ is realized as a lattice in $G$. For each $n$, the composition $\alpha _n \circ p_n$ is a homomorphism from $\Gamma$ to $P_n$, hence we obtain a homomorphism $\varphi : \Gamma \rightarrow P$ defined by $\varphi (b)_n \defeq  \alpha _n (p_n(b))$. Define the map $i: \Gamma \rightarrow G$ by
\[
i(b)\defeq  (b, \varphi (b)),
\]
for $b\in \Gamma$. The map $i$ is injective, and we show that it is a homomorphism. Clearly $i(1_\Gamma )=1_G$. For $b,c\in \Gamma$ and $n\in \Nb$, we have
\[
(b\circledast c)_n = b_n \alpha _n (p_n (b))(c_n) = b_n \varphi (b)_n(c_n) = b_n (\varphi (b)\cdot c)_n,
\]
and hence $b\circledast c = b \odot (\varphi (b)\cdot c)$, where $\odot$ denotes group multiplication in $B=\bigoplus _n B_n$. Therefore,
\begin{align*}
i(b)i(c)=(b, \varphi (b))(c,\varphi (c)) = (b\odot (\varphi (b)\cdot c),\varphi (b)\varphi (c)) &= (b\circledast c , \varphi (b\circledast c)) \\
&= i(b\circledast c) .
\end{align*}
The image $i(\Gamma )$, of $\Gamma$ in $G$, is discrete since $i(\Gamma )$ intersects the compact open subgroup $\{ 1 \} \rtimes P$ trivially. In addition, it is clear that $G= i(\Gamma )(\{ 1 \} \rtimes P)$, and therefore $i(\Gamma )$ is a cocompact lattice in $G$.

%Say that $\Gamma$ acts on $\Delta$ with finite orbits. Then $\Delta \rtimes \Gamma$ and $\Delta \times \Gamma$ are lattices in some locally compact group. Indeed, let $K$ be the closure of $\Gamma$ in $\Aut(\Delta)$. Then $\Delta \rtimes \Gamma$ and $\Delta\times \Gamma$ are lattices in $(\Delta\times \Gamma)\rtimes K$.

\subsection{Inner amenability }\label{ex:lattice invariant} We here give an example, due to Caprace, that shows inner amenability is not a lattice invariant, even in the case that we restrict to lattices in \tdlc groups.

For a prime $p$, let $F_p$ be the finite field with $p$ elements. We denote by $F_p[t]$, $F_p[[t]]$, and $F_p((t))$ the polynomials, formal power series, and formal Laurent series with a single indeterminate $t$ over $F_p$, respectively. Under the natural topology on $F_p((t))$, $F_p((t))$ is a local field of characteristic $p$. The formal power series ring $F_p[[t]]$ is a compact open subring of $F_p((t))$, $F_p[t^{-1}]$ is a discrete subring, and $F_p[[t]]F_p[t^{-1}]= F_p((t))$.

Forming the local field $F_p((t^{-1}))$, the ring $F_p[[t^{-1}]]$ is a compact open subring. The special linear group $SL_3(F_p[[t^{-1}]])$ acts on $F_p((t^{-1}))^3$, and it preserves the subgroup $F_p[[t^{-1}]]^3$. We thus obtain an action of $SL_3(F_p[[t^{-1}]])$ on $F_p((t^{-1}))^3 / F_p[[t^{-1}]]^3=F_p[t]^3$, where $F_p[t]$ is equipped with the discrete topology. We set
\[
G \defeq \left(SL_3(F_p[t^{-1}])\ltimes F_p((t))^3\right)\times \left( SL_3(F_p[[t^{-1}]])\ltimes F_p[t]^3\right)
\]
and equip $G$ with the product topology. With this topology $G$ is a \tdlc group.

 Define
\[
\Delta\defeq \left(SL_3(F_p[t^{-1}])\ltimes F_p[t^{-1}]^3\right)\times F_p[t]^3
\]
and
\[
\Gamma\defeq  SL_3(F_p[t^{-1}])\ltimes F_p[t,t^{-1}]^3.
\]

\begin{claim} $\Delta$ and $\Gamma$ are cocompact lattices in $G$.
\end{claim}
\begin{proof}
For $\Delta$, let $\psi:\Delta\rightarrow G$ by $((A,f),g)\mapsto ((A,f),(1,g))$. This map is injective, and the image $\psi(\Delta)$ is a cocompact discrete subgroup. Appealing to Theorem~\ref{thm:covolume}, $\Delta$ is a cocompact lattice in $G$.

For $\Gamma$, let $\psi:\Gamma\rightarrow G$ by $(A,f)\mapsto ((A,f),(A,f))$. This map is injective with discrete image, and
\[
\psi(\Gamma)\left(\{1\}\times F_p[[t]]^3\times SL_3(F_p[[t^{-1}]])\times \{1\}\right)=G.
\]
Using again Theorem~\ref{thm:covolume}, we conclude that $\Gamma$ is a cocompact lattice in $G$.
\end{proof}

The group $\Delta$ is inner amenable, because it has an infinite center.  On the other hand,  $\Gamma$ is not inner amenable. Indeed, if $m$ is a conjugation invariant mean on $\Gamma$, then since $SL_3(F_p[t^{-1}])$ is has Property (T), $m$ must concentrate on the set of elements of $\Gamma$ which commute with some finite index subgroup of $SL_3(F_p[t^{-1}])$. Every such element belongs to the center of $SL_3(F_p[t^{-1}])$, which is finite, so $m$ cannot be atomless. %This requires an ancillary result.

\subsection{Bounded conjugacy rank}\label{ex:rkcovol} This example shows that one cannot naively extend Theorem~\ref{thm:rkcovol} to all locally compact groups.

Take
\[
M\defeq
\begin{bmatrix}
    0   & 0 & 0 & -1 \\
   1    & 0 & 0& 3 \\
    0   & 1 & 0  & -3\\
   0    & 0 & 1  & 3
\end{bmatrix}
\]
The characteristic polynomial of $M$ is $p(\lambda)=\lambda^4-3\lambda^3+3\lambda^2-3\lambda +1$. One verifies that $p$ has two complex roots $z$ and $\ol{z}$ which lie on the unit circle but are not roots of unity and two distinct positive real roots $r$ and $s$ with $rs=1$. For any $n\in \Zb\setminus\{0\}$, it now follows that the matrix $M^n$ again has four distinct eigenvalues: $z^n$, $\ol{z}^n$, $s^n$, and $r^n$. Thus, $M^n$ cannot fix any $v\in \Zb^4\setminus \{\ol{0}\}$, and the action of $\Zb$ on $\Zb^4$ by powers of $M$ has no non-trivial finite orbit. Forming the semidirect product $\Zb^{4}\rtimes_M \Zb$ where $\Zb$ acts by powers of $M$, we infer that $\Zb^4\rtimes_M \Zb$ has a trivial $FC$-center. In view of Lemma~\ref{prop:rk_discrete}, $\rk_{BC}(\Zb^4\rtimes_M\Zb)=0$.

Recalling that taking powers of matrices extends to real number valued powers, we have an action of $\Rb$ on $\Rb^4$ by powers of $M$. As $M$ has two eigenvalues on the unit circle,
\[
V\defeq \{ x \in \Rb^4\mid \Rb. x \text{ is bounded}\}
\]
is a two dimensional subspace of $\Rb^4$.

We now form the semidirect power $G\defeq \Rb^4\rtimes_M \Rb$, and this group is a connected locally compact group. The subgroup $\Zb^4\rtimes_M\Zb$ is clearly a lattice in $G$ and $\rk_{BC}(\Zb^4\rtimes_M\Zb)=0$. On the other hand, we may find an irrational real number $\theta$ such that $M^{\theta}$ fixes $V$ pointwise. The subgroup $\Rb^4\rtimes_M\theta\Zb$ is then a cocompact subgroup of $G$ with $V$ in its center. Moreover, any natural way to extend the BC-rank to all locally compact groups must give $\Rb^4\rtimes_M\theta\Zb$ positive rank.

\subsection{ $\cFC$ elements}\label{ssec:pos_BC_rank_ex}

We finally give an example of a \tdlc group with positive bounded conjugacy rank but without non-trivial $\overline{\mathrm{FC}}$-elements.

Fix a prime $p$. For every prime $q$, let $\alpha_q$ denote the automorphism of the additive group $\Qb_p$ given by multiplication by $q$. Note that $\alpha _p$ has no nontrivial bounded orbits on $\Qb_p$ and in particular that the sequence $(\alpha _p ^{-n}(x))_{n\in \Nb}$ is unbounded for every nonzero $x\in \Qb_p$.

Fix a non-empty finite set of primes $S$ which are all distinct from $p$ and let $A$ denote the subgroup of $\mathrm{Aut}(\Qb_p)$ generated by the automorphism $\alpha _p$ along with the automorphisms $\alpha _q $ for $ q\in S$. The group $A$ is a free abelian group $A$ of rank $|S|+1$, because $A$ is isomorphic to the subgroup of the multiplicative group $\Qb^*$ of the rationals generated by $\{ p\} \cup S$. Let $A_S$ be the subgroup of $A$ generated by the elements $\alpha _q$ with $q\in S$. Since each prime other than $p$ is invertible in $\Zb_p$, the group $A_S$ may be viewed as a subgroup of the multiplicative group $\Zb_p^*$ of $\Zb_p$, and hence the action of $A_S$ on $\Qb_p$ has bounded orbits, since $\Zb_p^*$ acting on $\Qb_p$ has bounded orbits.

Form the semidirect product $G\defeq  \Qb_p\rtimes A$ where A has the discrete topology; the resulting group $G$ is a second countable \tdlc group. The action of the subgroup $\{ 0 \} \times A_S$ on $G$ by conjugation has bounded orbits. Indeed, the orbit of $(x,\alpha ) \in G$ is precisely $A_S(x)\times \{ \alpha \}$, which is bounded by the previous paragraph.  The group $G$ thus has BC-rank at least $\mathrm{rank}(A_S)=|S|$. Also, since $\Qb_p = \bigcup _n \alpha _p ^{-n} (\Zb_p)$, the group  G is compactly generated by the copy of $\Zb_p$ along with a finite generating set for $\{ 0 \} \times A$.

\begin{claim}
$G$ has no non-trivial $\overline{\mathrm{FC}}$ elements.
\end{claim}
\begin{proof}
 Suppose that $g\defeq  (x,\alpha )\in G$ is an $\overline{\mathrm{FC}}$-element, where $x\in \Qb_p$ and $\alpha \in A$. It must be the case that $x=0$ since otherwise the conjugacy class of $g$ is unbounded via conjugating $g$ by negative powers of $(0,\alpha _p )$. Thus, $g=(0,\alpha )$ for some $\alpha \in A$.

 Suppose toward a contradiction that $\alpha$ is not the identity automorphism, so there is some $y\in \Qb_p$ such that $\alpha (y)\neq y$. Since $y- \alpha (y)\neq 0$, the sequence $(\alpha _p^{-n}(y- \alpha (y) )_{n\in \Nb}$ is unbounded in $\Qb_p$. On the other hand, conjugating $g=(0,\alpha )$ by $(\alpha _p ^{-n}(y), 1)$ gives
\[
\begin{array}{rcl}
(\alpha _p ^{-n}(y), 1)(0,\alpha )(-\alpha _p ^{-n}(y), 1) & = & (\alpha _p ^{-n}(y)-\alpha (\alpha _p ^{-n}(y)) , \alpha ) \\
 &= & (\alpha _p ^{-n}(y-\alpha (y)) , \alpha ).
 \end{array}
\]
Thus, $g=(0,\alpha )$ has an unbounded conjugacy class in $G$, a contradiction. We conclude that $\alpha =1$ and that $g$ is trivial.
\end{proof}

%=========================== The bibliography===================================

\bibliographystyle{abbrv}
\bibliography{biblio}

\end{document}